\documentclass[10pt]{article}
\usepackage{amssymb}
\usepackage{amsthm}
\usepackage{graphicx}
\newtheorem{theorem}{Theorem}
\newtheorem{lemma}{Lemma}
\newtheorem{corollary}{Corollary}
\newtheorem{remark}{Remark}

\def\Frac#1#2{\frac{\displaystyle{#1}}{\displaystyle{#2}}}
\begin{document}
 \title{Uniform (very) sharp bounds for ratios of Parabolic Cylinder functions
}

\author{J. Segura\\
        Departamento de Matem\'aticas, Estadistica y 
        Computaci\'on,\\
        Universidad de Cantabria, 39005 Santander, Spain.\\
        javier.segura@unican.es
}

\date{ }

\maketitle
\begin{abstract}
Parabolic Cylinder functions (PCFs) are classical special functions with applications in many different fields. 
However, there is little 
information available regarding simple uniform approximations and bounds for these functions.  
We obtain very sharp bounds for the ratio 
$\Phi_n(x)=U(n-1,x)/U(n,x)$ 
and the double ratio $\Phi_n(x)/\Phi_{n+1}(x)$ in terms of elementary functions (algebraic or trigonometric) and prove the 
monotonicity of these ratios; bounds for $U(n,z)/U(n,y)$ are also made available. The bounds are very sharp as $x\rightarrow 
\pm \infty$ and $n\rightarrow +\infty$, and this simultaneous 
sharpness in three different directions explains
their remarkable global accuracy. Upper and lower elementary bounds are obtained which are able to produce several digits of accuracy
for moderately large $|x|$ and/or $n$.
\end{abstract}

%{\bf Keywords:} Parabolic Cylinder functions, bounds

%{\bf MSC2020:} 33C15, 26D07.

\section{Introduction}

Parabolic Cylinder functions (PCFs), solutions of the second order ODE
\begin{equation}
y''(x)-\left(\Frac{x^2}{4}+n\right)y(x)=0,
\end{equation}
are classical special functions (with Hermite functions as a particular case) which find applications in 
many scientific fields. In particular, the ratios of the recessive solution as $x\rightarrow  +\infty$, $U(n,x)$, 
appear in probabilistic contexts, for example as moments of certain distributions (see \cite{Read:2016:ALI} 
for an application in crystallography) and in the analysis of the Ornstein-Uhlenbeck process 
\cite{Siegert:1951:OTF,Alili:2005:ROT,Koch:2020:UBA}.

Despite their relevance, little is know regarding simple functional uniform approximations 
and bounds \cite{Segura:2012:OBF,Koch:2020:UBA}, which is in contrast to the vast information available on other functions like, for instance, 
modified Bessel functions (see \cite{Segura:2011:BFR,Hornik:2013:ABF,Baricz:2015:BFT,Ruiz:2016:ANT} 
and references cited therein). Recently one new monotonicity property was discovered 
for ratios of PCFs 
from stochastic considerations \cite{Koch:2020:UBA} 
that had not been proved by purely analytical
methods before; this fact illustrates that there is a clear deficit in the amount of analytical information on these functions. 

Here we propose to invert the process and to advance new 
properties, which supersede previous knowledge, using only basic analytical information, without recourse to 
more indirect arguments. We will not only prove the monotonicity property described in \cite{Koch:2020:UBA}, 
but we will obtain from this analysis new and very sharp bounds for ratios of parabolic cylinder functions. 

In particular we will obtain lower and upper bounds for the ratio $\Phi_n(x)=U(n-1,x)/U(n,x)$ that are very sharp in three 
different directions: $x\rightarrow \pm\infty$ and $n\rightarrow +\infty$. They are so sharp that, for instance,
 some of the bounds will display relative accuracies of order ${\cal O}(x^{-6})$ as $x\rightarrow +\infty$, 
 ${\cal O}(x^{-4})$ as $x\rightarrow -\infty$ and  ${\cal O}(n^{-2})$ as $n\rightarrow +\infty$\footnote{If the bound had relative 
error ${\cal O}(x^{-1})$ we would already say that it is sharp; then, we can say that our bounds are very sharp in 
three different directions, and even extremely sharp
as $|x|\rightarrow +\infty$.}.
The bounds are, in fact, tight enough to give an accurate estimation of $\Phi_n(x)$ for moderate $|x|$ and $n$ by using a very simple 
expression in terms of elementary functions (trigonometric or algebraic). It is expected that these approximations 
will play a role 
in future numerical algorithms for computing these special functions, in particular for improving the 
performance of the evaluation of $U(n,x)$ by backward recurrence and continued fraction evaluations \cite{Gil:CTR:2006}.

The techniques employed will be related to those considered for modified Bessel 
functions in \cite{Segura:2011:BFR} and later generalized in \cite{Segura:2012:OBF} (with application to PCFs), but they will go far 
beyond the possibilities of the analysis 
of Riccati equations presented in those papers, and they offer a new method for 
obtaining accurate information on the ratios of a large number of special functions.

%so much so that ... sharpness... monotonicity...

%One of the things is to prove monotonicity from first principles: just function properties, and to 
%go beyond. Beyond the optimal in some cases and in some sense we will later explain.

%This technique is applicable to many other cases  which are currently under investigation, but the PCF case
%is a simple and ideal model to test this technique.

The structure of the paper is as follows: In section \ref{props} we give the basic analytical information we will need 
of the PCF $U(n,x)$. In section \ref{ricatti} we briefly summarize the results in \cite{Segura:2012:OBF}; 
this will set the starting point and it will provide some necessary information for the next step. In addition the
 results of this section will suggest a very sharp bound for the double ratio $\Phi_n(x)/\Phi_{n+1}(x)$ which we
will be able to prove in section \ref{beyond}. In section
\ref{beyond} we discuss the new method and we will prove that the new results (of trigonometric or algebraic form) 
supersede previous results in all aspects.
We prove the monotonicity of the ratio $\Phi_n(x)$ and the double ratio $\Phi_n(x)/\Phi_{n+1}(x)$ and obtain
very sharp bounds for both ratios, from where bounds for the ratios $U(n,y)/U(n,z)$ can also be obtained.
Finally we illustrate numerically the sharpness of these bounds in section \ref{numerical}.

\section{Properties of the function $U(n,x)$}
\label{props}

In order to prove all the monotonicity properties and bounds in this paper we will only need the following 
three pieces of information:

\begin{enumerate}
\item{} The PCF $U(n,x)$ satisfies the following difference-differential system (see 12.8.2 and 12.8.3 of 
\cite{Temme:2010:PCF}):

\begin{equation}
\label{DDE}
\begin{array}{l}
U'(n,x)=\Frac{x}{2}U(n,x)-U(n-1,x),\\
\\
U'(n-1,x)=-\Frac{x}{2}U(n-1,x)-(n-1/2)U(n,x),
\end{array}
\end{equation}
and, as a consequence, they satisfy the three-term recurrence relation
\begin{equation}
\label{TTRR}
(n+1/2)U(n+1,x)+xU(n,x)-U(n-1,x)=0.
\end{equation}

\item{} As $x\rightarrow +\infty$ the function 
\begin{equation}
\Phi_n(x)=\Frac{U(n-1,x)}{U(n,x)}
\end{equation}
 is positive and increasing when $n>1/2$.

\item{} As $x\rightarrow +\infty$ the function 
\begin{equation}
W_n(x)=\left(n+\frac12\right)\Frac{\Phi_n(x)}{\Phi_{n+1}(x)}=\left(n+\frac12\right)\Frac{U(n+1,x)U(n-1,x)}{U(n,x)^2}
\end{equation}
 is positive and increasing when $n>1/2$ (the factor $n+1/2$ is introduced for later convenience).

\end{enumerate}

As we see, the required properties are very few, and the ideas should be easily applicable to a the wide range of functions
described in \cite{Segura:2012:OBF}, that is, for solutions of monotonic difference-differential linear systems. We will explore this in future papers.

The properties as $x\rightarrow +\infty$ are easy to check from the asymptotic expansions of $U(n,x)$
\cite[12.9.1]{Temme:2010:PCF}:
\begin{equation}
\label{xp}
U(n,x)\sim e^{-x^2/4} x^{-n-1/2}\displaystyle\sum_{s=0}^{\infty}(-1)^s\Frac{\left(\frac12 +n\right)_{2s}}{s! (2x^2)^s},\, 
 x\rightarrow+\infty .
\end{equation}
This expansion will also be useful for analyzing the sharpness of the bounds as $x\rightarrow +\infty$. Similarly, for analyzing the
sharpness of the bounds as $x\rightarrow -\infty$ we use the expansion \cite[11.2.23]{Temme:2015:AMFI}:
\begin{equation}
\label{xp2}
U(n,x)\sim \Frac{\sqrt{2\pi}e^{x^2/4} (-x)^{n-1/2}}{\Gamma\left(\frac12 +a\right)}
\displaystyle\sum_{s=0}^{\infty}\Frac{\left(\frac12 -n\right)_{2s}}{s! (2x^2)^s},\,  x\rightarrow -\infty .
\end{equation}

Using (\ref{xp})  we see that as $x\rightarrow +\infty$, 
\begin{equation}
\label{u+}
\Phi_n (x)\sim x\left[1+(n+1/2)x^{-2}
-(n+1/2)(n+3/2)x^{4}+{\cal O}(x^{-6})
\right],
\end{equation}
and the first term is enough to see that, indeed, $\Phi_n(x)$ is positive and increasing as $x\rightarrow +\infty$.

Now from (\ref{xp2}) we have, as $x\rightarrow -\infty$,
\begin{equation}
\label{u-}
\Phi_n (x)\sim -\Frac{n-1/2}{x}\left[1-(n-3/2)x^{-2}
+2(n-3/2)(n-2)x^{4}+{\cal O}(x^{-6})
\right].
\end{equation}

Also, combining (\ref{xp}) and (\ref{xp2})
\begin{equation}
\label{Wexpan}
W_n (x)\sim (n\pm 1/2)\left(1\mp x^{-2} 
+\left(\frac92 \pm 3n \right)x^{-4}+{\cal O}(x^{-6})\right),\,x\rightarrow \pm \infty,
\end{equation}
and therefore $W_n(x)$ is increasing and positive as $x\rightarrow +\infty$, as announced.

\section{Bounds from the Riccati equation and the recurrence relation}
\label{ricatti}

We start this section by briefly summarizing 
some of the bounds that were obtained in \cite{Segura:2012:OBF}, two of them rediscovered in \cite{Koch:2020:UBA}. We
prove them in a more straightforward way than in \cite{Segura:2012:OBF}, and we use this to motivate the more 
accurate methods to be described later.
Additionally we provide information on the sharpness of the bounds, to be compared with the new and tighter bounds of
section \ref{beyond}. Finally, a new (very sharp) bound
is motivated by the bounds in this section, which we will be able to prove true in section \ref{beyond}.

From (\ref{DDE}) we deduce that $\Phi_n(x)=U(n-1,x)/U(n,x)$ satisfies 
\begin{equation}
\label{ricat}
\Phi_n' (x)=\Phi_n (x)^2 -x\Phi_n(x) -(n-1/2).
\end{equation}

Therefore $\Phi_n(x)$ is one of the solutions of the Ricatti equation
$$
y'(x)=y(x)^2-x y(x)-(n-1/2).
$$

In our analysis, the nullclines of the differential equations 
play a major role; these are the curves where $y'(x)=0$. We will use the following lemma, which is immediate to
prove with graphical arguments.

\begin{lemma}
\label{lemma1}
Let $y(x)$ be a solution of $y'(x)=y(x)^2-x y(x)-a$, $a>0$, such that $y(+\infty)>0$ and $y'(+\infty)>0$, then
$y'(x)>0$ and $y(x)>\lambda_+(x)$ for all $x\in{\mathbb R}$, 
where $\lambda_+ (x)=(x+\sqrt{x^2+4a})/2$ is the positive nullcline of the Riccati equation.
\end{lemma}

\begin{proof}
The nullclines of the Riccati equation are $\lambda_{\pm}(x)$, with $\lambda_+ (x)=(x+\sqrt{x^2+4a})/2$ and
$\lambda_- (x)=-\lambda_+(-x)$ and then $y'(x)=(y(x)-\lambda_+ (x))(y(x)-\lambda_-(x))$; and because $y(+\infty)>0$ and $y'(+\infty)>0$
necessarily $y(x)>\lambda_+(x)$ for large enough $x$. 

But then, because $\lambda'_+(x)>0$ for all $x$, we have that $y(x)>\lambda_+(x)$ for all real $x$.  
In other to see this, we can follow the graph of $y(x)$ starting 
from a large enough value of $x$, $x_\infty$, such 
that $y(x_\infty)>\lambda_+ (x_\infty)$ and following the curve in the direction of decreasing $x$ as follows. 

In the first place 
it is not possible that a value $x_c<x_\infty$ exists such that $y(x_c)=\lambda_+ (x_c)$ (and then $y'(x_c)=0$) 
starting from $y(x_\infty)>\lambda_+ (x_\infty)$, 
because we would be approaching the nullcline from above and then $y'(x_c)$ should be larger than $\lambda'_+ (x_c)>0$, and there is
a contradiction. On the other hand $y(x)$ stays finite for all real $x$ (and continuous and differentiable): 
as said it can not cross the nullcline and the 
only possibility would be to go to $+\infty$ for some $x_p<x_{\infty}$ 
but this is not possible because as long as $y(x)>\lambda_+(x)$ we have
that $y'(x)>0$ and therefore the values of $y(x)$ decrease as $x$ decreases.  
\end{proof}

\begin{theorem}
\label{cota1}
For all real $x$ and $n>1/2$ the ratio $\Phi_n(x)=U(n-1,x)/U(n,x)$ is increasing as a function of $x$ and
\begin{equation}
\label{bcota1}
\Phi_n (x)>\Frac{1}{2}\left(x+\sqrt{x^2+4n-2}\right).
\end{equation}
\end{theorem}
\begin{proof}
Considering (\ref{xp}) we have that $\Phi_n(x)=x(1+{\cal O}(x^{-2}))$ as $x\rightarrow +\infty$. 
Therefore $\phi_n(x)$ satisfies the hypotheses of Lemma \ref{lemma1} with $a=n-1/2$.
\end{proof}

\begin{remark}
Considering (\ref{u+}) and (\ref{u-}) we conclude that the bound is sharp both as $x\rightarrow \pm \infty$, as
 the first term in the expansion as $x\rightarrow \pm\infty$ coincides. It is, in fact, the only bound of
the type $\alpha x+\sqrt{\beta x^2 +\gamma}$ satisfying these conditions, and in this sense it is the best possible
lower bound.
\end{remark}

\begin{remark}
All the bounds in this paper are sharp as $n\rightarrow +\infty$. See section \ref{numerical} for details.
\end{remark}

Now we obtain two upper bounds combining the previous result with the use of the recurrence relation (\ref{TTRR})

\begin{corollary} For $n>-1/2$ and all real $x$
\begin{equation}
\label{rebo1}
\Phi_n(x)<\Frac{1}{2}\left(x+\sqrt{x^2+4n+2}\right).
\end{equation}
\end{corollary} 

\begin{proof}
We write the recurrence relation (\ref{TTRR}) in terms of $\Phi_n(x)$ as follows
 (backward recurrence)
\begin{equation}
\label{backward}
\Phi_n (x)=x+(n+1/2)/\Phi_{n+1}(x).
\end{equation}
Now using Theorem \ref{cota1} for $\Phi_{n+1}(x)$ we obtain the bound.
\end{proof}

\begin{remark}
Care must be taken when using this to write a lower bound for $\Phi_n(x)^{-1}=U(n,x)/U(n-1,x)$, for $n\in (-1/2,1/2)$ 
because in this case $\Phi_n(x)$ becomes negative for negative $x$. 
\end{remark}

\begin{remark}
The bound (\ref{rebo1}) is sharp as $x\rightarrow +\infty$ and the first two terms in the expansion (\ref{xp}) are reproduced. Contrarily, 
the bound is not sharp as $x\rightarrow -\infty$, as it gives $-(n+1/2)/x+{\cal O}(x^{-3})$. It is, however, the bound of
the form $\alpha x+\sqrt{\beta x^2 +\gamma}$ with the highest order of approximation as $x\rightarrow +\infty$ and such
that it is ${\cal O}(x^{-1})$ as $x\rightarrow -\infty$. 
\end{remark}

\begin{corollary} For $n>3/2$ and all real $x$
\begin{equation}
\label{rebo2}
\Phi_n(x)<\Frac{1}{2}\Frac{n-1/2}{n-3/2}\left(x+\sqrt{x^2+4n-6}\right).
\end{equation}
\end{corollary}
\begin{proof}
We rewrite the three-term recurrence relation as (forward recurrence)
\begin{equation}
\label{forward}
\Phi_n (x)=\Frac{n-1/2}{-x+\Phi_{n-1}(x)},
\end{equation}
and apply Theorem \ref{cota1} to $\Phi_{n-1}(x)$.
\end{proof}

\begin{remark}
The bound (\ref{rebo1}) is sharp as $x\rightarrow -\infty$ and the first two terms in the expansion (\ref{xp}) are reproduced. Contrarily, 
the bound is not sharp as $x\rightarrow +\infty$, as it gives $x(n-1/2)/(n+3/2)+{\cal O}(x^{-1})$. It is, however, the bound of
the form $\alpha x+\sqrt{\beta x^2 +\gamma}$ with the highest order of approximation as $x\rightarrow -\infty$ and such
than it is ${\cal O}(x)$ as $x\rightarrow +\infty$. 
\end{remark}

The previous inequalities can be combined to obtain bounds for the double ratio $W_n(x)=(n+1/2)\Phi_n(x)/\Phi_{n+1}(x)$. It 
is convenient to define
$$h_{\alpha,\beta}(x)=(n-\beta)\Frac{x+\sqrt{4(n-\alpha)+x^2}}{x+\sqrt{4(n-\beta)+x^2}},$$
%=\frac14 (x+\sqrt{4(n-\alpha)+x^2})(-x+\sqrt{4(n-\beta)+x^2}),$$
which satisfies $h_{\alpha,\beta}(x)=h_{\beta,\alpha}(-x)$ and
\begin{equation}
\label{exph}
h_{\alpha,\beta}(x)=(n-\beta)\left[1+\Frac{\alpha-\beta}{x^2}\left(1-\Frac{3n-\alpha-2\beta}{x^2}\right)\right]
+{\cal O}(x^{-6})
\end{equation}
as $x\rightarrow +\infty$ and the same expansion with $\alpha$ interchanged with $\beta$ as $x\rightarrow -\infty$.

We have
\begin{theorem}
\label{pri1}
The following holds for all real $x$ and $n>1/2$, except that the last inequality only holds for $n>3/2$:
$$
\Frac{1}{n+3/2}h_{\frac12,-\frac32}(x)<
\Frac{1}{n+1/2}W_n(x)<1<\Frac{1}{n-1/2}W_n(x)<\Frac{1}{n-3/2}h_{\frac32,-\frac12}(x)
$$
\end{theorem}

\begin{remark}
The central inequalities $\Frac{1}{n+1/2}W_n(x)<1<\Frac{1}{n-1/2}W_n(x)$ reappeared in \cite{Koch:2020:UBA}, but 
they were already proved in \cite{Segura:2012:OBF}. The first and last inequalities in this chain of inequalities 
were not correctly estated in 
\cite[Thm. 11]{Segura:2012:OBF}: the value $x=0$ was erroneously set.
\end{remark}

\begin{remark}
The first two inequalities in Thm. \ref{pri1} (at the left) 
are sharp as $x\rightarrow +\infty$ while the two bounds at the right are sharp
as  $x\rightarrow -\infty$. The extreme bounds are sharper as two terms of the corresponding asymptotic expansions coincide, 
while for the central bounds only the dominant term is given. None of these bounds is simultaneously sharp as $x\rightarrow 
\pm \infty$, differently from the bounds we describe in the next section.
\end{remark}

We notice that the function $W_n(x)$ has a sigmoidal 
shape similar to the functions $h_{\alpha,\beta}(x)$, $\alpha>\beta$ and that the selection 
$\alpha=1/2$, $\beta=-1/2$ gives the exact limit values as $x\rightarrow \pm \infty$. We propose the following result valid
for real $x$ and $n>1/2$ which we will be able to prove in section \ref{beyond}:
\begin{equation}
\label{conjeturado}
W_n(x)>h_{\frac12,-\frac12}(x).
\end{equation}

\begin{remark}
The bound $h_{\frac12,-\frac12}(x)$ is the best possible upper bound of the form $h_{\alpha,\beta}(x)$,
 because it is sharp both as $x\rightarrow \pm \infty$; in fact,
 the first two terms in the expansions are correct.
\end{remark}

\section{Beyond the Riccati bounds}
\label{beyond}

The previous analysis clearly suggests that $W_n (x)$ should be an increasing function of $x$. However, from the previous inequalities
alone it does not seem possible to prove this result, which is know to be true and has been proved by indirect arguments 
\cite{Koch:2020:UBA}. Here we give a direct proof of this
result and, as a by-product of this analysis, we obtain the tightest available bounds for the ratios and double rations of Parabolic 
Cylinder functions. We expect that the same ideas can be applied to other monotonic special functions.

The idea is to construct a differential equation involving $W_n (x)$ and to analyze the nullclines, as done before. The analysis
is not so straightforward as before, but it will be rewarding.

The starting point is the recurrence relation (\ref{TTRR}), which we multiply by $U(n-1,x)/U(n,x)^2$ yielding
$$
(n+1/2)\Frac{U(n+1,x)U(n-1,x)}{U(n,x)^2}+x\Frac{U(n-1,x)}{U(n,x)}-\left(\Frac{U(n-1,x)}{U(n,x)}\right)^2=0,
$$
which in our notation is

\begin{equation}
W_n(x)=\Phi_n (x) (\Phi_n (x)-x)
\end{equation}

We can simplify the Riccati equation (\ref{ricat}) by substituting 
\begin{equation}
\label{varphi}
\Phi_n(x)=\Frac{x}{2}+\phi_n(x)
\end{equation}
and we get:
\begin{equation}
\phi'_n(x)=\phi_n(x)^2-V_n(x),\, V_n(x)=\Frac{x^2}{4}+n.
\end{equation}
Now we have 
\begin{equation}
\label{defw}
W_n(x)=\phi_n (x)^2-\Frac{x^2}{4}, 
\end{equation}
and, because $\phi_n(x)>0$ for all real $x$ and $n>1/2$ (Theorem \ref{cota1}),
\begin{equation}
\phi_n(x)=+\sqrt{\frac{x^2}{4}+W_n(x)}.
\end{equation}
Taking the derivative of (\ref{defw}):
\begin{equation}
\label{3deg}
W'_n (x)=2\phi_n(x)\phi'_n (x)-x/2=2\left(\phi_n(x)^3-V_n(x)\phi_n(x) -\Frac{x}{4}\right).
\end{equation}

We can now substitute $\phi_n(x)=\sqrt{x^2/2+W_n(x)}$ and we would have a differential equation relating 
$W'_n(x)$ with $W_n(x)$, and the analysis of the nullclines of the equation together with the asymptotic 
properties of the function $W_n(x)$ could be investigated in order to prove the monotonicity of the function
and to obtain a bound. The analysis would be similar to the one considered before for $\Phi_n(x)$ in section 
\ref{ricatti} with the
difference that we don't have a Riccati equation now. It is however simpler to study the sign of (\ref{3deg}) in terms
of the values of $\phi_n(x)$ and to map the resulting results to the $x-W_n$ plane. For this idea, the first step
is to solve for the values of $\phi_n(x)$ that make $W'_n(x)=0$. The structure of these solutions is discussed next.

\subsection{Properties of the nullclines}

In this subsection, we prove five lemmas in relation with Eq. (\ref{3deg}) and its nullclines that are needed for proving 
the main result of this paper (Theorem \ref{main}). 

\begin{lemma}
The cubic equation
\begin{equation}
\label{cubic}
\lambda(x)^3-V_n(x)\lambda_n(x) -\Frac{x}{4}=0
\end{equation}
has, for all real $x$ and $n>1/2$, three distinct real solutions.
\end{lemma}
\begin{proof}
We recall that given a cubic equation $ax^3+bx^2+cx+d=0$, the equation has three different real solutions if the discriminant, given
by 
$$
\Delta=18abcd-4b^3 d +b^2 c^2-4 a c^3 -27 a^2 d^2,
$$
is positive, while it has only one real solution (and two complex conjugate solutions) if $\Delta<0$. Now with $a=1$, $b=-V_n(x)$ and 
$c=-x/4$ we have
\begin{equation}
\label{px}
16\Delta (x)=16\left\{-4(-V_n(x))^3-27\Frac{x^2}{16}\right\}=x^6+12 x^4 n+(48 n^2-27) x^2+64 n^3 .
\end{equation}
That $\Delta (x)>0$ 
for all real $x$ is obvious for $|n|>3/4$ because all the coefficients are positive in this case, but this is also
 true for all $|n|>1/2$ as can be easily checked by computing the discriminant of the cubic equation $f(z)=16\Delta(\sqrt{z})=0$. 
We have now $a=1$, $b=12 n$, 
$c=48 n^2-27$, $d=64$, and the discriminant of the third degree polynomial $f(z)$ is
$
\tilde{\Delta}=314928\left(-n^2+\frac14\right),
$
which is negative if $|n|>1/2$, meaning that $f(z)=\Delta (\sqrt{z})$ has only one real root for $z$; 
such root must be negative, because
 on account of Descartes' rule of signs,   $f(z)$ has exactly one negative real root  if $|n|<3/4$,
 which gives $x$ purely imaginary. 
This proves
that $\Delta(x)$ has no real roots if $|n|>1/2$ and therefore $\Delta(x)>0$ for all real $x$.
\end{proof}

\begin{lemma}
\label{sols3}
Denoting the three (real) solutions of (\ref{cubic}) 
by $\lambda_n^- (x) <\lambda_n^0 (x) <\lambda_n^+ (x)$, we have, for all real $x$ and $n>1/2$, 
$\lambda_n^+ (x)>0$, $\lambda_n^{-}(x)=-\lambda_n^{+}(-x)$ and $\lambda_n^0(x)=-\lambda_n^0(-x)$, 
with $\mbox{sign}(\lambda_n^0(x))=
\mbox{sign}(-x)$. 

The positive solution can be written:
$$
\begin{array}{ll}
\lambda_n^+ (x)&=f_n(x)\cos
\left(
\Frac{1}{3}\arccos\left(\Frac{x}{f_n(x)^3}\right)
\right),\,f_n(x)=\sqrt{\Frac{x^2+4n}{3}}.
\end{array}
$$

\end{lemma}

\begin{proof}
We start giving an explicit expression for the roots using the well known formula in terms of trigonometric functions, which is
\begin{equation}
\label{solus}
\lambda (x)=\Frac{2}{\sqrt{3}} \sqrt{V_n(x)}\cos
\left(
\frac13\arccos\left(\Frac{3\sqrt{3}x}{8V_n(x)^{3/2}}\right)+\psi
\right),
\end{equation}
with the three values $\psi=0,\pm 2\pi/3$. Of course, the solutions 
are continuous and differentiable whenever the argument of the arc-cosine is smaller 
than $1$ in absolute value, and this condition
is equivalent to the positivity of the discriminant, which we have already proved for all real $x$ and $n>1/2$.

Now, because all the three roots are real and simple, and 
considering Descartes' rule of signs, we conclude that there is only one possible
positive solution for $x>0$, which is $\lambda_n^+ (x)$, while the other two must be negative; similarly, for $x<0$ there is
only one negative solution and therefore two positive solutions.

Then $\lambda_n^+ (x)$ must be positive for all real $x$. Now, because the cubic equation is invariant under the simultaneous 
changes $x\rightarrow -x$ and $\lambda(x)\rightarrow -\lambda(x)$, if $\lambda(x)$ solves the equation also does $-\lambda (-x)$.
Therefore $-\lambda_n^+(-x)$ is a second solution. And there is a third solution which must satisfy $\lambda(x)=-\lambda(-x)$, and 
therefore
changes sign (and is zero at $x=0$). Then we have $\lambda_n^0(x)=-\lambda_n^0(-x)$, 
while $\lambda_n^-(x)=-\lambda_n^+(-x)$. 
There is only one solution which is positive at $x=0$, which is $\lambda_n^+(x)$, given by (\ref{solus}) with $\psi=0$.
\end{proof}

\begin{lemma} For all real $x$ and $n>1/2$,
\label{lamme}
$\lambda_n^+(x)^2>x^2/4$, $\lambda_n^-(x)^2>x^2/4$ and $\lambda_n^0(x)^2<x^2/4$.
\end{lemma}

\begin{proof}
We make the substitution $\lambda(x)=\mu(x)-x/2$ in (\ref{cubic}) and we have 
$$
2\mu(x)^3-3 x\mu (x)^2+(x^2-2n) \mu (x)+x (n-1/2)=0,
$$
which for $x<0$ has only one positive root, on account of Descartes' rule of signs; this solution has to be $\mu(x)=\lambda_n^+ (x) 
+x/2$. Therefore $\lambda_n^0 (x)+x/2<0$ for $x<0$ and, because $\lambda_n^0 (x)$ is positive for $x<0$, we have $|\lambda_n^0 (x)|<|x|/2$
if $x<0$; and since $\lambda_n^0(x)=- \lambda_n^0(-x)$ this inequality is true for all real $x$.

Finally, with the substitution $\lambda(x)=\mu(x)+x/2$ in (\ref{cubic}) and with the same analysis using Descartes' rule of signs we 
conclude that $\lambda_n^+(x)-x/2>0$ for $x>0$, which together with the previous condition for $x<0$ ($\lambda_n^+(x)+x/2>0$) 
gives $\lambda_n^+(x)>|x|/2$. And
because $\lambda_n^-(x)=-\lambda_n^+(-x)$ we also have that $\lambda_n^-(x)>|x|/2$.
\end{proof}

There is only one additional lemma that we will need in the analysis of the nullclines for (\ref{3deg}), namely:

\begin{lemma}
\label{wcre}
$w_n (x)=\lambda_n^+(x)^2-\Frac{x^2}{4}$ is increasing as a function of $x$.
\end{lemma}

\begin{proof}
We have 
\begin{equation}
\label{wp}
\omega'_n (x)=2 \lambda_n^+ (x)\lambda_n^{+ \prime} (x)-\Frac{x}{2}.
\end{equation}
Now differentiating
\begin{equation}
\label{3ot}
\lambda_n^+(x)^3-V_n(x)\lambda_n^+(x) -\Frac{x}{4}=0
\end{equation}
we obtain
$$
-\Frac{x}{2}\lambda_n^+(x)-\frac14+\lambda_n^{+ \prime} (x)(3\lambda_n^+(x)^2-V_n(x))=0
$$
and then
\begin{equation}
\label{lap}
\lambda_n^{+ \prime} (x)=\Frac{1+2x\lambda_n^+(x)}{4(3\lambda_n^+(x)^2-V(x))}.
\end{equation}
Now, inserting this in (\ref{wp})
$$
\omega'_n (x) =\Frac{\lambda_n^+ (x)(1+2x\lambda_n^+(x))}{2(3\lambda_n^+(x)^2-V(x))}-\Frac{x}{2}=\frac12\Frac{
\lambda_n^+ (x)+x(V(x)-\lambda_n^+(x)^2)}
{3\lambda_n^+(x)^2-V(x)},
$$
and using (\ref{3ot}) $V(x)-\lambda_n^+ (x)^2=-x/(4\lambda_n^+ (x))$, and so
$$
\omega'_n (x) = \frac12\Frac{
\lambda_n^+ (x)^2-x^2/4}{\lambda_n^+(x)(3\lambda_n^+(x)^2-V(x))}.
$$

The denominator is positive because $\lambda_n^+ (x)^2-x^2/4>0$. With respect to the
denominator 
$$
3\lambda_n^+(x)^2-V(x))>\frac34 x^2-\Frac{x^2}{4}-n=\Frac{x^2}{2}-n$$
which is certainly positive for large enough $|x|$. A possible change of sign would be
 caused by a singularity in $\lambda_n^+ (x)$ (see Eq. (\ref{lap})) but this does not occur for $n>1/2$, because 
$\lambda_n^+ (x)$ is well defined and differentiable, as we saw.
\end{proof}

From the study of the nullclines, and in particular of the only active nullcline for our problem, we can 
prove the next result, that will be used in establishing our main result.
\begin{lemma}
\label{lemmaprin}
Let $y(x)$ satisfy the differential equation
\begin{equation}
\label{ypz}
y'(x)= 2\left(z(x)^3-\left(\Frac{x^2}{4}+n\right)z(x)-\Frac{x}{4}\right),\,n>1/2
\end{equation}
where
\begin{equation}
\label{yz}
z(x)=+\sqrt{\frac{x^2}{4}+y(x)}.
\end{equation}
If $y(x)$ is positive and increasing as $x\rightarrow +\infty$ then 
$$
z(x)>\lambda_n^+(x)=f_n(x) \cos\left(
\frac13\arccos\left(\Frac{x}{f_n(x)^3}\right)\right),\,f_n(x)=\sqrt{\Frac{x^2+4n}{3}}
,
$$
$$
y(x)>\lambda_n^+(x)^2-\Frac{x^2}{4}
$$
and $y'(x)>0$ for all real $x$.
\end{lemma}
\begin{proof}
Because $y(+\infty)>0$ we have that $z(x)>x/2$ for large $x$. Then, denoting as before by 
$\lambda_n^-(x)<\lambda_n^0(x)<\lambda_+(x)$ the roots of 
$\lambda(x)^3-\left(\Frac{x^2}{4}+n\right)\lambda(x)-\Frac{x}{4}=0$, by lemma \ref{lamme}
we have that $z(+\infty)>\lambda_n^0(x)$ because $z(x)>|x|/2$.

We are left with two possibilities, that $z(+\infty)>\lambda_n^+(+\infty)$ or the contrary, 
but because
$$
y'(z)=2(z(x)-\lambda_n^+ (x))(z(x)-\lambda_n^0 (x))(z(x)-\lambda_n^- (x))
$$
and by hypotheses $y'(+\infty)>0$ then necessarily $z(+\infty)>\lambda_n^+(+\infty)$.

But this implies that 
$y(+\infty)=z(+\infty)^2-\Frac{x^2}{4}>w_n(+\infty)$ where $w_n(x)=\lambda_n^+(x)^2-x^2/4$.

In other words, the function $y(x)$ lies above the curve $y=w_n(x)$, which is the active nullcline of 
the differential equation ($\lambda_n^0(x)$ and $\lambda_n^-(x)$ play no role because $z(x)=\sqrt{x^2/4+y(x)}>|x|/2$). 
But from lemma \ref{wcre} we know that $w'_n(x)>0$, and therefore 
the fact that $y(+\infty)>w_n (+\infty)$ and that $y'(x)>0$ for values of $y(x)>w_n (x)$ 
implies, by the same arguments used in lemma \ref{lemma1}, that $y(x)>w_n (x)$ for all real $x$ and therefore
that  $y'(x)>0$ for all $x$ and, finally, $z(x)=\sqrt{x^2/4+y(x)}> \sqrt{x^2/4+w_n (x)}=\lambda_n^+ (x)$.
\end{proof}

\subsection{Main results: trigonometric and algebraic uniform lower bounds}

We are now in the position to prove the main result of the paper, which provides very sharp bounds for the ratios and double 
ratios of parabolic cylinder functions. 

\begin{theorem}
\label{main}
Let $\Phi_n(x)=\Frac{U(n-1,x)}{U(n,x)}$ and $W_n(x)=(n+1/2)\Frac{\Phi_n (x)}{\Phi_{n+1}(x)}$ then the following holds for all real $x$ and $n>1/2$.

\begin{enumerate}
\item{} $\Phi_n (x)>0$, $\Phi'_n (x)>0$, $\Phi''_n(x)=W'_n(x)>0$.
\item{} $\Phi_n(x)>\varphi_n(x)=\Frac{x}{2}+\lambda_n^+ (x)$.
\item{} $W_n(x)=(n-1/2)+\Phi'_n (x)>w_n(x)=\lambda_n^+ (x)-x^2/4$.
\end{enumerate}
\end{theorem}

\begin{proof}
That  $\Phi_n (x)>0$, $\Phi'_n (x)>0$ was proved in Theorem \ref{cota1}. The rest is
 a consequence of Lemma \ref{lemmaprin} because, on account of (\ref{Wexpan}), we have that $W_n (+\infty)>0$ and 
$W'_n (+\infty)>0$. $W_n(x)$ plays the role of $y(x)$ in the lemma, and $\phi_n(x)=\Phi_n(x)-\Frac{x}{2}$ plays the role of $z(x)$.
\end{proof}

\begin{remark}
The new bounds in theorem \ref{main} are sharp as $n\rightarrow +\infty$ (as all the bounds in this paper) and extremely sharp as
$|x|\rightarrow +\infty$, particularly as $x\rightarrow +\infty$. 

All the terms shown in (\ref{u+}) coincide 
with the expansion of the bound $\varphi_n(x)$ of the previous theorem, and we have that as $x\rightarrow +\infty$.
%$$
%\Frac{1}{x}\left(\Phi_n (x)-\varphi_n (x)\right)=\Frac{2n+1}{x^6}+{\cal O}(x^{-8}),
%$$
$$
\Frac{\Phi_n (x)}{\varphi_n (x)}-1=\Frac{2n+1}{x^6}+{\cal O}(x^{-8}),
$$
which is amazingly sharp, while as $x\rightarrow -\infty$
%$$
%-\Frac{x}{n-1/2}\left(\Phi_n (x)-\varphi_n (x)\right)=\Frac{2}{x^4}+\Frac{-20 n+24}{x^6}+{\cal O}(x^{-8}),
%$$
$$
\Frac{\Phi_n (x)}{\varphi_n (x)}-1=\Frac{2}{x^4}+{\cal O}(x^{-6}),
$$
which is not so sharp, but very sharp in any case.

Similarly, the bound for $W_n(x)$ is also very sharp as $x\rightarrow \pm\infty$ and we have that
$$
w_n(x)=(n\pm 1/2)(1-x^{-2})+(\pm 3n^2 +4n \pm \frac54)x^{-4}+{\cal O}(x^{-6})
$$
and then $W_n/w_n(x)-1=(2n\pm 1)x^{-4}+{\cal O}(x^{-6})$.
\end{remark}

Next, we are proving that these lower bounds are tighter for all $x$ 
than the upper bounds that were proved in the previous section, and even than the bound (\ref{conjeturado}), 
which we will prove later.

\begin{theorem}
The lower bound for $\Phi_n(x)$ of Theorem \ref{main} is sharper than the bound of Theorem \ref{cota1}, that is:
$$
\lambda_n^+(x)>\frac12 \sqrt{x^2+4n-2},\,n>1/2,\,x\in{\mathbb R}
$$
\end{theorem}
\begin{proof}
We prove that the bound  (\ref{bcota1}) corresponds to a curve which lies below the nullcline $y=w_n(x)$ of (\ref{ypz}), which
 in the $z(x)$ variable is $z=\lambda_n^+ (x)$. To see this, we   
substitute $z(x)$ in (\ref{ypz}) by the corresponding bound of $\phi_n(x)=\Phi_n(x)-x/2=\sqrt{x^2+4n-2}/2$ 
and we check that $y'(x)<0$.

Indeed, setting $z(x)=\sqrt{x^2+4n-2}/2$ in (\ref{3deg}) we have
$$
\frac12 y'(x)=-\frac12 z(x) -\Frac{x}{4}<0,
$$
which completes the proof.
\end{proof}

Next, we prove that the previously conjectured lower bound of (\ref{conjeturado}) is smaller than $w_n(x)$, and therefore it is
indeed a lower bound for $W_n(x)$.

\begin{theorem}
\label{best}
For all real $x$ and $n>1/2$
$$
W_n(x)>w_n(x)>w_n^{(a)}(x)=h_{\frac12,-\frac12}(x),
$$
and $w_n^{(a)}(x)$ is the best possible bound of the form $h_{\alpha,\beta}(x)$. 
\end{theorem}
\begin{proof}
To see this we check that if we take $y(x)=w_n^{(a)}(x)$ in (\ref{yz}) then
we have $y'(x)<0$ in (\ref{ypz}), which means that the curve $y=w_n^{(a)}(x)$ lies below the nullcline $y=w_n(x)$, and therefore
$w_n^{(a)}(x)<w_n(x)$; and because we have proved that $W_n(x)>w_n(x)$ then $W_n(x)>w_n(x)>w_n^{(a)}(x)$.

Indeed, taking $z(x)=h_{1/2,-1/2}(x)$ we have $y (x)=\sqrt{h_{\frac12,-\frac12}(x)+\frac{x^2}{4}}$ and from this and using (\ref{ypz}) we will have that $y'(x)<0$ if
\begin{equation}
\label{eq1}
\sqrt{h_{\frac12,-\frac12}(x)+x^2/4}\left(h_{\frac12,-\frac12}(x)-n\right)-\Frac{x}{4}<0.
\end{equation}
Now, we substitute $h_{\frac12,-\frac12}(x)$ inside the square root by its supremum in ${\mathbb R}$, which is $(n+1/2)$, and
denoting 
$$
g_n(x)=\Frac{1}{2}\sqrt{x^2+4n+2}\left[(n+1/2)\Frac{x+\sqrt{x^2+4n-2}}{x+\sqrt{x^2+4n+2}}-n\right]-\Frac{x}{4},
$$
if we prove that $g_n(x)<0$ then inequality (\ref{eq1}) is proved. And after some elementary algebra
$$
g_n(x)=-\Frac{(n+1/2)(x^2+4n-\sqrt{x^2+4n+2}\sqrt{x^2+4n-2})}{2(x+\sqrt{x^2+4n+2})},
$$
from where is it obvious that $g_n(x)<0$ for all real $x$ and $n>1/2$.

\end{proof}

\begin{remark}
The algebraic bound $w_n^{(a)}(x)$ is nearly as sharp as $w_n(x)$, because $w_n(x)/w_n^{(a)}(x)-1=(2n\pm 1) x^{-4}+{\cal O}(x^{-6})$. 
\end{remark}

Using the very sharp algebraic bound in Thm. \ref{best} we can obtain easily a very sharp algebraic bound for $\Phi_n(x)$:

\begin{corollary}
\label{last}
For all real $x$ and $n>1/2$ the following holds
$$
\Phi_n(x)>\varphi_n^{(a)}(x)=
\Frac{x}{2}+\sqrt{\Frac{x^2}{4}+w_n^{(a)}(x)}=
\Frac{x}{2}+\displaystyle\sqrt{\Frac{x^2}{4}+\left(n+\frac12\right)\Frac{x+\sqrt{x^2+4n-2}}{x+\sqrt{x^2+4n+2}}}
$$
\end{corollary}

\begin{remark}
The bound of Corollary \ref{last} has similar sharpness as the bound of theorem \ref{main} (but the relative difference with the
value of $\Phi_n(x)$ is approximately two times bigger). Indeed, we have that as $x\rightarrow +\infty$
$$
\Frac{\Phi_n (x)}{\varphi_n^{(a)}(x)}-1=2\Frac{2n+1}{x^6}+{\cal O}(x^{-8}),
$$
while as $x\rightarrow -\infty$
$$
\Frac{\Phi_n (x)}{\varphi_n^{(a)}(x)}-1=\Frac{4}{x^4}+{\cal O}(x^{-6}).
$$
\end{remark}

\begin{corollary} 
\label{integ}
For $n>1/2$ and $z>y$ we have:
\begin{equation}
\Frac{U(n,z)}{U(n,y)}<\exp\left(-\displaystyle\int_{y}^z\lambda_n^+(x)\,dx\right)
\end{equation}
and
\begin{equation}
\Frac{U(n,z)}{U(n,y)}<\exp\left(-\displaystyle\int_{y}^z\sqrt{\Frac{x^2}{4}+w_n^{(a)}(x)}\,dx\right).
\end{equation}
\end{corollary}
\begin{proof}
Using (\ref{DDE}) we have that $\Phi_n(x)-\Frac{x}{2}=-U'(n,x)/U(n,x)$.
Therefore, from the bound of Corollary \ref{last} we have
$$
-\Frac{U'(n,x)}{U(n,x)}>\sqrt{\Frac{x^2}{4}+w_n^{(a)}(x)},
$$
and the same can be written for the trigonometric bounds, changing $w_n^{(a)}(x)$ by $w_n(x)=\lambda_n^+(x)^2-x^2/4$. 
Now the results follow after integration.
\end{proof}

\begin{remark}
The bounds in Corollary \ref{integ} are given in terms of integrals of elementary functions, but the integrals can not be 
expressed in terms of elementary functions. Still, explicit bounds in terms of elementary functions can be extracted by taking 
into account that $w_n(x)$ and $w_n^{(a)}(x)$ are increasing functions. Then, for instance, using the trigonometric bound
$$
\Frac{U(n,z)}{U(n,y)}<\exp\left(-\displaystyle\int_{y}^z\sqrt{\Frac{x^2}{4}+w_n(x)}\,dx\right)<
\exp\left(-\displaystyle\int_{y}^z\sqrt{\Frac{x^2}{4}+w_n(y)}\,dx\right).
$$
And the last bound, though no so sharp, can be expressed in terms of elementary functions and is certainly sharper than the
bound that can be extracted from Theorem \ref{cota1}, which is equivalent to replacing $w_n (y)$ with $w_n(-\infty)=n-1/2$.
\end{remark}

\subsection{Upper bounds}

In the analysis we have obtained very sharp lower bounds as $x\rightarrow \pm \infty$ 
for $\Phi_n(x)$, both of trigonometric ($\varphi_n(x)$, Thm. \ref{main}) and
algebraic form ($\varphi_n^{(a)}(x)$, Corollary \ref{last}) 
and similarly very sharp bounds $w_n(x)$ (Thm. \ref{main}) and $w_n^{(a)}(x)$ (Thm. \ref{best}) 
for the double ratio 
$W_n(x)$,
which 
drastically improve the bounds known so far. As done in section \ref{ricatti}, we can obtain upper bounds for 
$\Phi_n(x)$ from the lower
bound by applying the recurrence relation both in the forward (\ref{forward}) and the backward (\ref{backward}) direction,
 and then from these upper bounds we can obtain also upper bounds for $W_n(x)$ using that $W_n(x)=\Phi_n(x)-x^2/4$.  

We don't discuss these eight  additional bounds in detail (for $\Phi_n(x)$ and $W_n(x)$ and starting from the bounds 
$\varphi_n(x)$ or $\varphi_n^{(a)}$ and with two
different recursion directions), but only mention that the sharpness is preserved differently for
the forward and the backward application of the recurrence (as also described in section \ref{ricatti}). 

In the first place, we consider the bound from the backward recurrence, starting from $\varphi_n(x)$, and we have:
\begin{corollary} For all real $x$ and $n>-1/2$ the following holds:
$$\Phi_n(x)<\bar{\varphi}_n (x)=x+\Frac{n+1/2}{\varphi_{n+1}(x)}.$$

As $x\rightarrow +\infty$
$$
\Frac{\Phi_n (x)}{\bar{\varphi}_n (x)}-1={\cal O}(x^{-8}),
$$
and as $x\rightarrow -\infty$
$$
\Frac{\Phi_n (x)}{\bar{\varphi}_n (x)}-1={\cal O}(x^{-2}).
$$

\end{corollary}

We conclude that $\bar{\varphi}_n (x)$ 
is sharper than $\varphi_n(x)$ as $x\rightarrow +\infty$ but less sharp for $x\rightarrow -\infty$.
Differently from the backward recurrence, for the forward recurrence the sharpness as $x\rightarrow +\infty$ is maintained and 
improved as $x\rightarrow -\infty$, although the range of validity is restricted to $n>3/2$:

\begin{corollary}
 For all real $x$ and $n>3/2$ the following holds:
$$\Phi_n(x)<\tilde{\varphi}_n (x)=\Frac{n-1/2}{-x+\varphi_{n-1}(x)}.$$

As $x\rightarrow +\infty$
$$
\Frac{\Phi_n (x)}{\tilde{\varphi}_n (x)}={\cal O}(x^{-4}),
$$
and as $x\rightarrow -\infty$
$$
\Frac{\Phi_n (x)}{\tilde{\varphi}_n (x)}-1={\cal O}(x^{-8}).
$$
\end{corollary}

\begin{remark}
Bounds for the ratio $U(n,y)/U(n,z)$ can also be established from the bounds in this section, similarly as done in 
Corollary \ref{integ}. We skip the details for such bounds.
\end{remark}

\section{Numerical illustration of the sharpness of the bounds} 
\label{numerical}

To end the description of the new bounds, we estimate as a function of $n$ the maximum errors for the bounds for all $x$ in 
${\mathbb R}$. This will give information on the global accuracy of the bounds. For this, we are estimating the errors 
as $n$ becomes large using a simplification: we estimate that the maximum error takes place at $x=0$. Let us recall 
that we are considering bounds which are very sharp as $x\rightarrow \pm \infty$.  
A more detailed analysis is possible using uniform asymptotics for
the PCFs (see \cite[12.10.25]{Temme:2010:PCF}), but for our purpose it is enough with this simple analysis and 
we will see how this simplification is in fact quite reasonable for estimating these errors.

Because $W_n(0)=\Phi_n (0)^2$, the analysis for the errors for the bounds for $W_n(x)$ follow easily from those for 
$\Phi_n(0)$, and the relative accuracy of these bounds will be approximately twice the error for the bounds of
$\Phi (x)$.

The values at $x=0$ (see \cite[12.2.6]{Temme:2010:PCF}) are:

$$
\Phi_n (0)=\sqrt{2}\Frac{\Gamma\left(\frac34+\frac{n}{4}\right)}{\Gamma\left(\frac14 +\frac{n}{2}\right)}
$$
and
$$
\varphi_n^{(0)} (0)=\sqrt{n-1/2},\,
\varphi_n (0)=\sqrt{n},\,\varphi_n^{(a)} (0)=(n^2-1/4)^{1/4},
$$
where $\varphi_n^{(0)}(x)$ denotes the not so sharp bound of Theorem \ref{cota1}. Now, we estimate the relative accuracy as
$n$ becomes large by expanding in powers of $n^{-1}$.

We have
$$
\Frac{\Phi_n (0)}{\varphi_n^{(0)}(0)}-1=\Frac{1}{4n}+{\cal O}(n^{-2}),
$$ 
$$
\Frac{\Phi_n (0)}{\varphi_n (0)}-1=\Frac{1}{16n^2}+{\cal O}(n^{-4}).
$$ 
$$
\Frac{\Phi_n (0)}{\varphi_n^{(a)}(0)}-1=\Frac{1}{8n^2}+{\cal O}(n^{-4}).
$$ 

We observe that the new bounds are also sharper  than the previous bounds $\varphi_n^{(0)}(x)$ as $n$ becomes large, 
with a relative deviation decreasing quadratically. This means 
that, for instance, for $n>10$ $\varphi_n(x)$ is an estimation for $\Phi_n (x)$ with at least two correct digits for all real 
$x$, and decreasing at least as ${\cal O}(x^{-4})$ as $|x|$ becomes large. 
As $n$ becomes larger the situation of course improves and, for instance, for $n>2500$ eight digits accuracy is attained 
for all real $x$. To our knowledge, such degree of uniformity and accuracy is without precedent in the estimation of non-trivial
special functions depending on two parameters. The reason for this good behavior is the fact that the bounds are (very) sharp
 in three different directions:
 $x\rightarrow \pm\infty$ and $n\rightarrow +\infty$.

For the upper bounds $\bar{\varphi}_n (x)$ and $\tilde{\varphi}_n (x)$ the errors also decrease quadratically (they are roughly twice
as large). Notice also that because we have upper and lower bounds which are similarly sharp as $x\rightarrow \pm\infty$ and 
$n\rightarrow +\infty$, it is possible to estimate the accuracy of the estimation by comparing the upper with the lower bound.

In order to illustrate this we show in Figure \ref{fig1} the plot of the curves 
\begin{equation}
\label{region}
\epsilon=\Frac{\tilde{\varphi}_n (x)}{\varphi_n (x)}-1
\end{equation}
in the $(x,n)$ plane and for three different values of $\epsilon$ which approximately correspond to three, four and five correct
digits. 
Notice that $\tilde{\varphi}_n (x)/\varphi_n (x)-1<\epsilon$ in the unbounded region outside the curve (\ref{region}).

\begin{figure}[tb]
\vspace*{0.8cm}
\begin{center}
\begin{minipage}{6cm}
\centerline{\includegraphics[height=6cm,width=8cm]{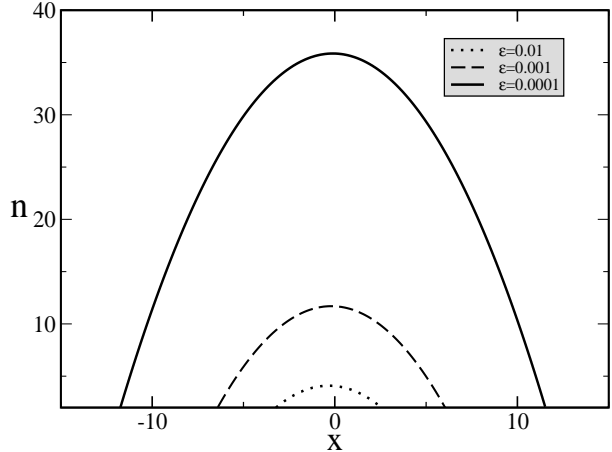}}
\end{minipage}
\end{center}
\caption{Curves $\epsilon=\Frac{\tilde{\varphi}_n (x)}{\varphi_n (x)}-1$ for three different values of $\epsilon$}
\label{fig1}
\end{figure}

We conclude by commenting that bounding a function in terms of more elementary functions can be useful for extracting 
information from expressions which involve such functions without the need to compute them. This is maybe the most typical use of function bounds. However, in this case the bounds are so accurate that they can provide by themselves fairly accurate approximations for 
even moderate values of the variables. Because of the good accuracy of these bounds in unbounded domains in the $(x,n)$ plane, they 
will be useful in numerical algorithms for computing parabolic cylinder functions. For instance, the bounds for $U(n-1,x)/U(n,x)$ 
can be used for estimating the tail of the contined fraction representation for this ratio when $x>0$, 
accelerating in this way the convergence. A similar
 procedure was described in \cite{Segura:2011:BFR} for modified Bessel functions, and in the present case the new bounds are sharper and accurate in larger domains.

\section*{Acknowledgements}
The author acknowledges support from Ministerio de Ciencia e Innovaci\'on, project PGC2018-098279-B-I00
(MCIU/AEI/FEDER, UE).

\bibliography{parabolic}

\end{document}